\documentclass[12pt, reqno]{amsart}
\usepackage{amsmath, amsthm, amscd, amsfonts, amssymb, graphicx, color}
\usepackage[bookmarksnumbered, colorlinks, plainpages]{hyperref}

\newtheorem{theorem}{Theorem}[section]
\newtheorem{lemma}[theorem]{Lemma}

\newtheorem{corollary}[theorem]{Corollary}
\theoremstyle{definition}
\newtheorem{definition}[theorem]{Definition}
\newtheorem{example}[theorem]{Example}

\theoremstyle{remark}
\newtheorem{remark}[theorem]{Remark}
\numberwithin{equation}{section}

\begin{document}
\setcounter{page}{1}

\title[The Roots and Links in a Class of $M$-Matrices]{The Roots and Links in a Class of $M$-Matrices}

\author[X.-D. Zhang]{Xiao-Dong Zhang$^1$$^{*}$ }

\address{$^{1}$ Department of Mathematics and MOE-LSC, Shanghai Jiao Tong University, Shanghai 200240, P.R.China.}
\email{\textcolor[rgb]{0.00,0.00,0.84}{xiaodong@sjtu.edu.cn}}


\dedicatory{This paper is dedicated to Professor Tsuyoshi Ando }



\date{Received: xxxxxx; Revised: yyyyyy; Accepted: zzzzzz.
\newline \indent $^{*}$ Corresponding author}

\begin{abstract}
In this paper, we discuss exiting roots of sub-kernel transient matrices
 $P$ associated with a class of $M-$ matrices which are related to  generalized ultrametric   matrices. Then the results   are  used to describe completely all links of the class of matrices in terms of structure of  the supporting tree.
\end{abstract} \maketitle

\section{Introduction}

\noindent
Let $I$ be a finite set and $|I|=n$. $U=(U_{ij}, i,j \in I)$ is
\textit{ultrametric matrix} if it is symmetric, nonnegative and satisfies
the ultrametric inequality
\begin{equation*}
U_{ij}\ge \displaystyle\min \{U_{ik}, U_{kj} \}\ \ \ {\rm for \ \
 all  } \ \ i,j,k\in I.
 \end{equation*}
 The ultrametric matrices have an  important  property that if $U$
 is nonsingular ultrametric matrix, then the inverse of $U$ is
  row and column diagonally dominant $M-$ matrix (see \cite{martinez1994} and \cite{nabben1994}).
  A construction also was given in
  \cite{nabben1994} to describe all such ultrametric matrices.
  Later, nonsymmetric ultrametric matrices were independently defined
  by McDonald, Neumann, Schneider and
Tsatsomeros
   in \cite{mcdonald1995}  and
Nabben and Varge in   \cite{nabben1995}.i.e.,  \textit{nested
block form(NBF) matrices} and \textit{generalized  ultrametric(GU)
matrices.}   After a suitable permutation, every GU matrix can be
put in NBF.   They satisfy ultrametric inequality and are
described with dyadic trees in \cite{mcdonald1995},
\cite{nabben1995} and \cite{dell2000}.  On the inequality of $M-$ matrices and inverse $M-$ matrices, Ando in \cite{ando1979} presents many nice and excellent inequalities which play an key role in the nonnegative matrix theory.  Zhang \cite{zhang2003} characterized  equality cases in Fisher, Oppenheim and Ando inequalities.  For more detail information
on inverse $M-$ and $Z-$ matrices, the reader is referred  to
 \cite{Johnson2011}, \cite{martinez2003}, \cite{martinez2004}, \cite{mcdonald1998}, \cite{zhang2004} and the references in there.
 In this paper, we follows closely the global frame work
and notation on generalized ultrametric matrices supplied by
Dellacherie, Mart\'{i}nez and Mart\'{i}n in
\cite{dell2000}.
  Recently, Nabben was motivated by the result of Fiedler in
  \cite{fiedler2000} and defined a  new class of matrices:
  $\mathcal{U}$-matrices (see Section 2) which satisfy
  ultrametric inequality and are related to GU matrices. There is
  a common characterization in these matrices that if they are nonsingular, then
  their inverses are column diagonally dominant $M-$matrices.
 For each $\eta\ge \eta(U)=\displaystyle\max\{ (U^{-1})_{ii},\  i\in
 I\}$, define matrix $P=E-\eta^{-1}U^{-1},$ where $E$ is
 the identity matrix. Then $P$ is
 sub-Markov kernel: $P_{ij}\ge 0,$ for all $i,j\in I$ and
 ${\textbf{1}}^tP\le {\textbf{1}}^t$ (entry-wise), where $\textbf{1}$ is the column vector of all ones.
 Therefore
  $$ \eta U=(E-P)^{-1}=\sum_{m\ge 0}P^m $$
  and $U$ is proportional to the potential matrix associated to the
  transient kernel $P$. Since $P_{ij}>0$ if and only if
  $(U^{-1})_{ij} < 0$ for $i\neq j$, the existence of links
  between different points does not depend on $\eta$, while the
  condition $P_{ii}>0$ depends on the value of $\eta$. Define the
  potential vector $\mu=\mu_U$ associated to $U$ by
  $\mu:=U^{-1}\textbf{1}$ and its total mass
  $\overline{\mu}:=\textbf{1}^t\mu$. Note that the following
  equivalence holds
  \begin{equation*}
  \mu_i>0 \Leftrightarrow (U^{-1}\textbf{1})_i>0 \Leftrightarrow
  (P\textbf{1})_i<1.
  \end{equation*}
  Every $i$ satisfying this property is called  an {\it exiting root} of
  $U$(or of $P$). The set of them is denoted by
  $\mathcal{R}:={\mathcal{R}}_U$. The {\it potential vector} $\nu:=\nu_U$
  associated to $U^t$ is given by $\nu:=(U^t)^{-1}\textbf{1}$ and
  $\overline{\nu}:=\textbf{1}^t\nu$. Notice that $
  \overline{\mu}=\overline{\nu}$, since
  $\textbf{1}^t\nu=\textbf{1}^t(U^t)^{-1}
  \textbf{1}=\textbf{1}^t\mu.$

  Our main results in this paper are to
  characterize the following properties (which do not depend on
  $\eta$)  `` $i$ is a exiting root of $P$ and $P^t$"; and `` link of $P$, i.e.,
   for a given couple
  $i\neq j,$ whether $P_{ij}> 0$ for $ \mathcal{U}$-matrices". These
  properties and other related problems were completely investigated  for
symmetric ultrametric matrices and  GU matrices in \cite{dell1996}
and
  \cite{dell2000}, respectively. In Section 2, we revisit $ \mathcal{U}$-
  matrices by means of dyadic tree and give some preliminary
  results which are very useful. In Section 3, we describe exiting root of $P$ and $P^t$ with
  associated trees. In Section 4, we characterize completely the links of $P$.

\section{$ \mathcal{U} $ matrices}

 \noindent A {\it tree} $(T,\mathcal{J})$ is a finite unoriented and connected
 acyclic
 graph. For $(t,s)\in T\times T, \ t\neq s$, there is a unique
 path geod$(t,s)$ of minimum length, which is called the {\it
 geodic}
 between $t$ and $s$, while geod$(t,t)=\{t\}$ which is of length
 $0$.  Sometime, we use geod$(t,s)$ to  stand for its edge set.  Fixed $r\in T,$ we call it the {\it  root} of tree $T$. If $s\in
 geod(t,r)$, we denote $s\preceq  t$, which is a partial
 order relation on $T$. For $s, t\in T$, $s\wedge t=\displaystyle
 \sup \{v, v\in geod(s,r)\cap geod(t,r)\}$ denotes the closest
 common {\it ancestor} of $s$ and $t$. The set of {\it successors} of $t$ is
 $s(t)=\{s\in T, \ s \succeq t,\ (s,t)\in \mathcal{J}\}$.
 Then $I({\mathcal{J}})=\{i\in T, \ s(i)=\emptyset\}$ is the set of
{\it  leaves} of the tree $T$. A tree is said {\it  dyadic} if
$|s(t)|=2$ for
 $t\notin I({\mathcal{J}})$. The successors of $t$ are denoted by
 $t^-$ and $t^+$.  For $t\in T$, the set $L(t):=\{i\in I( { \mathcal{J}}), t\in
 geod(i,r)\}$ completely  characterizes $t$. Hence we can identify $t$ and
 $L(t)$. In particular, $r$ is identified with $L(r)=I(
 {\mathcal{J}})$ and $i\in I( { \mathcal{J}})$ with the singleton
 ${i}$. Hence we can assume that each vertex of $T$ is a subset of
 the set of leaves $I({ \mathcal{J}})$. The distinction between the
 roles of $L$, as $L\in T $ (mean that $L$ is a vertex of tree $T$) and
 $L\subseteq I$ (mean that $L$ is regarded as the subset of $I( { \mathcal{J}})$
  corresponding to the vertex of $T$), will be clear in
 the context when we use them.
By the above notations and definition of $GU$ matrices in
\cite{dell2000}, The definition of   $\mathcal{U}$ matrices in
\cite{nabben2000} may be restated in the following way

\begin{definition} $U=(U_{ij}: i,j\in I)$ is a
$\mathcal{U}$ matrix if there exists a dyadic tree
$(T,{\mathcal{J}})$ with fixed a root $r$ and a leaf $n\in I$, and
nonnegative real vectors $
\overrightarrow{{\mathbf{\alpha}}}=(\alpha_t: t\in T)$, $
\overrightarrow{{ \mathbf{\beta} }}=(\beta_t:t\in T)$ satisfying

{\rm{(i).}}  $I=I( \mathcal{J})$,
$\overrightarrow{{\mathbf{\alpha}}}|_I=\overrightarrow{{
\mathbf{\beta} }}|_I ;$ and $\alpha_t=\alpha_{t\wedge n}$ for
$t\in r^+, t\notin I$;

{\rm{(ii).}} $\alpha_t\le \beta_t$ for $t\in T$;

{\rm{(iii).}} $\overrightarrow{{\mathbf{\alpha}}}$ and
$\overrightarrow{{ \mathbf{\beta} }}$ are $\preceq$- increasing,
i.e., $t\preceq s$ implies $\alpha_t\le \alpha_s$ and $\beta_t\le
\beta_s$;

{\rm{(iv).}} $t^+\in geod(r,n)$ for $t\in geod(r,n)$ and $t\neq
n$; and $\alpha_t=\beta_t$ for $t\in geod(r,n)$.

{\rm{(v). }} $U_{ij}=\alpha_t $ if $(i,j) \in (t^-,t^+)$ and
$U_{ij}=\beta_s$ if $(i,j)\in (t^+,t^-)$, where $t=i\wedge j$ and
$s=\displaystyle\max\{ i\wedge j, i\wedge n \}$;
$U_{ii}=\alpha_i=\beta_i$ for $i\in I$.

 \noindent We say that $(T,\mathcal{J})$ support $U$ and $U$ is  $\mathcal{U}$ associated with tree
 $(T,\mathcal{J})$.

\end{definition}

 It is easy to show  that this definition is equivalent to
Definition 2.1 in   \cite{nabben2000}. Observe that for each $L\in
T$, the matrix  $U|_{L\times L}$ is  either  GU or $ \mathcal{U} $
matrix, where the GU matrix consistent with the definition of GU
matrix in \cite{dell2000}. The
  tree supporting it, denoted by $(T|_L,{\mathcal{J}}_L)$, is the
  restriction of $(T,\mathcal{J})$ on  $L$ and the associated vectors which
   are the restrictions of
  $\overrightarrow{{\mathbf{\alpha}}}$ and $\overrightarrow{{
\mathbf{\beta} }}$ on $T|_L$. In other words,
$(T|_L,{\mathcal{J}}_L)$ is the subtree of $(T,\mathcal{J})$ with
the root $L$ and the leaves set $L$. The potential vectors and the
exiting roots of $U|_L, U^t|_L$ are denoted, respectively by
 $\mu_L,\ \nu_L,\ {\mathcal{R}}_L, \ {{ \mathcal{R}}^t}_L$. The
 sub-kernel corresponding to $U|_L, U^t|_L$ is denoted by $P^L$
 , $(P^t)^L$.
 If $U$ is nonsingular $\mathcal{U}$ matrix, it can be shown that
 $U|_L$ is also nonsingular GU or $ \mathcal{U}$ matrix by Schur
 decomposition and inductive argument.

We now introduce the following relation $ \leq_{\mathcal{J}}$ in
the set of leaves $I$.  For $i\neq j$, we say $i <_{\mathcal{J}}
j$ if $i\in t^-, j\in t^+$ with $t=i\wedge j$.
 Assume that $I=\{1,2,\cdots,n \}$. By permuting I, we can suppose
 $\leq_{\mathcal{J}}$ is the usual relation $\leq$ on $I$.
  Therefore, we will assume that this is standard presentation of
  $\mathcal{U}$ matrices in this paper. In the other words, Let $U\in \mathcal{U}$ and
  $I=I^-\cup I^+.$ Denote $J:=I^-$ and $K:=I^+$.
   Thus
      $$U=\left(\begin{array}{cc}
   U_J & \alpha_I\textbf{1}_J\textbf{1}_K^t\\
   b_K\textbf{1}_J^t& U_K
   \end{array}\right),$$
 where  $\alpha_I=\displaystyle \min\{ U_{ij}: i,j\in I\}$ and
 $b_K=U_Ke_K$ with $e_K=(0,\cdots,0, 1)^t$ unit vector, i.e., $b_K$
 is the last column of $U_K$. Note that in here, $U_J$ is GU matrix and  $U_K$ is still
$\mathcal{U}$ matrix also,  which has a similar $2\times 2$ block
structure, and its the first diagonal block is a special GU
matrix.
 We begin  with the following theorem in which we  re-prove some
 known result in \cite{nabben2000}.
\begin{theorem}\cite{nabben2000}
\label{schur decomposition} If  $U$ is nonsingular $ \mathcal{U}$
matrix, then

{\rm{(i).}} $\alpha_I\overline{\mu}_J< 1$ and
$$U^{-1}=\left(\begin{array}{rr}
C& D\\
E& F
\end{array}
\right),$$ where

\begin{eqnarray*} C=&
 U_J^{-1}+ \frac{\alpha_I}{1-\alpha_I\overline{\mu}_J}\mu_J\nu_J^t,
  \ \ \ & D=\frac{-\alpha_I}{1-\alpha_I\overline{\mu}_J}\mu_J\nu_K^t,\\
 E=& \frac{-1}{1-\alpha_I\overline{\mu}_J}e_K\nu_J^t,
\ \ \ &F=U_K^{-1}+\frac{\alpha_I \overline{\mu}_J
}{1-\alpha_I\overline{\mu}_J}e_K\nu_K^t. \end{eqnarray*}

{\rm{(ii)}}.
$$\mu_I=\left(\begin{array}{c}
\frac{1-\alpha_I\overline{\mu}_K}{1-\alpha_I\overline{\mu}_J}\mu_J\\
\mu_K-\frac{\overline{\mu}_J(1-\alpha_I\overline{\mu}_K)}{1-\alpha_I\overline{\mu}_J}e_K
\end{array}\right); \ \ \ \nu_I=\left(\begin{array}{l}
0\\
\nu_K
\end{array}
\right).$$

{\rm{ (iii).}} $ \overline{\mu}_I=\overline{\mu}_K$.

{\rm{(iv):}} $(\mu_I)_i\ge 0$ for $ i=1,2,\cdots,n-1.$

 {\rm{ (v)}}. $(\nu_I)_i=0$ for
 $i=1,2,\cdots, n-1$; and $(\nu_I)_n=\overline{\mu}_I=\frac{1}{U_{nn}}. $
\end{theorem}
\begin{proof}
Since $U$ is nonsingular, $U_J$ is nonsingular  GU matrix. By
Theorem 3.6(i) in \cite{mcdonald1995},
$\alpha_J\overline{\mu}_J\le 1$, where $\alpha_J$ is smallest
entry in $U_J$. Hence  $  \alpha_I\overline{\mu}_J\le
\alpha_J\overline{\mu}_J\le 1$. Suppose that
$\alpha_I\overline{\mu}_J= 1$, by theorem 3.6(ii) in
\cite{mcdonald1995}, $U_J$ has a row whose entries are all equal
to $\alpha_I$. Noting that the last row whose entries are equal to
$U_{nn},$ there  are two rows which are proportional, which
implies $U$ is singular, a contradiction. Therefore
$\alpha_I\overline{\mu}_J< 1$.
 By Schur decomposition and the inverse of matrix formula,
it is not difficult to show that the rest of (i) holds. Since
$\overline{\mu}_J=\overline{\nu}_J$ and
$\overline{\mu}_K=\overline{\nu}_K$,
\begin{eqnarray*}
C\textbf{1}_J+D\textbf{1}_K &=&
U_J^{-1}\textbf{1}_J+\frac{\alpha_I}{1-\alpha_I\overline{\mu}_J}\mu_J\nu_J^t\textbf{1}_J
+\frac{-\alpha_I}{1-\alpha_I\overline{\mu}_J}\mu_J\nu_K^t\textbf{1}_K\\
&=&\frac{1-\alpha_I\overline{\mu}_K}{1-\alpha_I\overline{\mu}_J}\mu_J,
\end{eqnarray*}
\begin{eqnarray*}
E\textbf{1}_J+F\textbf{1}_K &=&
\frac{-1}{1-\alpha_I\overline{\mu}_J}e_K\nu_J^t\textbf{1}_J+
U_K^{-1}\textbf{1}_K+\frac{\alpha_I \overline{\mu}_J
}{1-\alpha_I\overline{\mu}_J}e_K\nu_K^t\textbf{1}_K\\
&=&
\mu_K-\frac{\overline{\mu}_J(1-\alpha_I\overline{\mu}_K)}{1-\alpha_I\overline{\mu}_J}e_K,
\end{eqnarray*}
\begin{eqnarray*}
\textbf{1}_J^tC+\textbf{1}_K^tE&=&
\textbf{1}_J^tU_J^{-1}+\frac{\alpha_I}{1-\alpha_I\overline{\mu}_J}
\textbf{1}_J^t\mu_J\nu_J^t+ \textbf{1}_K^t
\frac{-1}{1-\alpha_I\overline{\mu}_J}e_K\nu_J^t \\
&=& 0,
\end{eqnarray*}
\begin{eqnarray*}
\textbf{1}_J^tD+\textbf{1}_K^tF&=&
\frac{-\alpha_I}{1-\alpha_I\overline{\mu}_J}\textbf{1}_J^t\mu_J\nu_K^t+
\textbf{1}_K^t U_K^{-1}+\textbf{1}_K^t\frac{\alpha_I
\overline{\mu}_J }{1-\alpha_I\overline{\mu}_J}e_K\nu_K^t \\
&=& \nu_K^t.\end{eqnarray*}
 So (ii) holds. Furthermore,
$$ \overline{\mu}_I=
\frac{1-\alpha_I\overline{\mu}_K}{1-\alpha_I\overline{\mu}_J}\overline{\mu}_J+
\overline{\mu}_K-\frac{\overline{\mu}_J(1-\alpha_I\overline{\mu}_K)}{1-\alpha_I\overline{\mu}_J}
=\overline{\mu}_K.$$ Thus (iii) holds. Since
$\frac{1}{U_{nn}}e_I^tU ={\textbf{1}^t}$,
  $\nu_I=\frac{1}{U_{nn}}e_I$ which implies
$\overline{\mu}_I=\overline{\nu}_I=\frac{1}{U_{nn}}$.  By (iii),
we have $1-\alpha_I\overline{\mu}_K=1-\alpha_I\overline{\mu}_I\ge
1-\frac{\alpha_I}{U_{nn}}\ge 0$. Hence   it is easy  to show that
(iv) and (v) hold
 by using the induction on the dimension of $U$.
\end{proof}

\section{Exiting roots of P}

\noindent In order to characterize the exiting roots of $P,$ we introduce
some notations and symbols. Let $U$ be a $ \mathcal{U}$ matrix
with supporting tree $(T, \mathcal{J})$ and fixed a root $r$ and a
leaf $n$ . For $i\in I( \mathcal{J}),$ denote by $ N_i^+=\{ L\in
T: L\preceq i,\ \alpha_L=\alpha_i\}$  and $ N_i^-=\{ L\in T:
L\preceq i,\ \beta_L=\beta_i\}$. Now we can construct the set
$\Gamma^t$: for $L\notin geod(r,n),$  $(L,L^-)\in \Gamma^t$ if and
only if there exists a $i\in L^+$ such that $L\in N_i^-; $
$(L,L^+)\in \Gamma^t$ if and only if there exists a $i\in L^-$
such that $L\in N_i^+ $. For $L\in geod(r,n),$ $(L,L^-)\in
\Gamma^t$ and $(L,L^+)\notin \Gamma^t$.
 \begin{theorem}
 \label{root of transpose of P}
Let $U$ be nonsingular $ \mathcal{U}$ . Then

{\rm{(i)}}. $ {\mathcal{R}}_I^t=\{n \}.$

{\rm{(ii).}} For $L\in T$, $i\in {\mathcal{R}}_L^t $ if and only
if $ geod (i,L)\cap \Gamma^t=\emptyset.$

 \end{theorem}
 \begin{proof}
 (i) follows from Theorem \ref{schur decomposition} (v).  We prove the assertion (ii) by
 using an induction on the dimension $n$  of $U$. It is trivial for
 $n=1,2$. Assume that
 $$U=\left(\begin{array}{cc}
   U_J & \alpha_I\textbf{1}_J\textbf{1}_K^t\\
   b_K\textbf{1}_J^t& U_K
   \end{array}\right).$$
    If $L\subseteq I^-$, then $U|_L=(U_J)|_L$ is  GU matrix. By
    Theorem 3 in \cite{dell2000} and $(r,r^-)\notin geod(i,L),$
    the assertion (ii) holds. If $L\subseteq I^+$, then $U|_L=(U_K)|_L$
    is still $\mathcal{U}$  matrix. By the hypothesis
and $(r,r^+)\notin geod(i,L)$, $i\in {\mathcal{R}}_L^t $ if and
only if $ geod (i,L)\cap \Gamma^t=\emptyset.$ If $L=I$, then for
$i\neq n$, $(i\wedge n, (i\wedge n)^-)\in \Gamma^t$ and
$geod(i,L)\cap \Gamma^t\neq \emptyset$; for $i=n$, $ geod(i,L)\cap
\Gamma^t=\emptyset.$ Hence by (i), $L\in T$, $i\in
{\mathcal{R}}_L^t $ if and only if $ geod (i,L)\cap
\Gamma^t=\emptyset.$
 \end{proof}
   In order to describe the exiting of $P$, we  construct the set
   $\Gamma$:  For each $L\in T$, $ (L,L^-) \in \Gamma$,  $ (L,L^+) \in \Gamma$, if and only
   if there exists  $i\in L^+$, $i\in L^-$, such that $L\in
   N_i^+$, $L\in N^-$, respectively.
   \begin{theorem}
   \label{ root of  P}
    Let $U=(U_{ij}, i,j\in I)$ be  a nonsingular  $\mathcal{U}$
    matrix. Then

    {\rm{(i).}} $n\neq i\in \mathcal{R}$ if and only if
    $geod(i,I)\cap \Gamma=\emptyset$.

{\rm{(ii)}}. If $L\in T$, then $n\neq i\in {\mathcal{R}}_L$ if and
only if
    $geod(i,I)\cap \Gamma=\emptyset$.
\end{theorem}
\begin{proof} (i) We prove the assertion by the induction. It is
trivial for $n=1,2$.   We assume that
 $$U=\left(\begin{array}{cc}
   U_J & \alpha_I\textbf{1}_J\textbf{1}_K^t\\
   b_K\textbf{1}_J^t& U_K
   \end{array}\right),$$
 where  $\alpha_I=\displaystyle \min\{ U_{ij}: i,j\in I\}$,
 $b_K=U_Ke_K$ with $e_K=(0,\cdots,0,1)^t$ unit vector, $U_J$ is GU
 matrix and $U_K$ is  $ \mathcal{U}$ matrix.  By  Theorem \ref{schur
 decomposition}(ii), we have
$$\mu_I=\left(\begin{array}{c}
\frac{1-\alpha_I\overline{\mu}_K}{1-\alpha_I\overline{\mu}_J}\mu_J\\
\mu_K-\frac{\overline{\mu}_J(1-\alpha_I\overline{\mu}_K)}{1-\alpha_I\overline{\mu}_J}e_K
\end{array}\right).
$$
 Now we consider the following two cases.

 {\bf Case 1:} $i\in I^-$. Then
  $i\in \mathcal{R}$ if and only if $i \in {\mathcal{R}}_J$ and $
  1-\alpha_I\overline{\mu}_K> 0 $ by  Theorem
    \ref{schur decomposition} (ii).  Note that $
  \overline{\mu}_I=\overline{\mu}_K$ from Theorem
     \ref{schur decomposition} (iii).
  $
  1-\alpha_I\overline{\mu}_K= 0 $  if and only if $
\alpha_I=\alpha_n$  if and only if $(I,I^-)\in
    \Gamma$ because it follows from the definition of $\Gamma$, and
  if  there exists an $n\neq q\in I^{+}$ such that $I\in N_q^+$ which
  implies $ \alpha_I=\alpha_q$ by definition of 2.1(i).
   Therefore each entries of the
  $q-$th row is $\alpha_q$  which implies that $U$ is singular. Hence $
  1-\alpha_I\overline{\mu}_K> 0 $ if and only if $(I,I^-)\notin
  \Gamma$. By the inductive hypothesis,
  $n\neq i\in \mathcal{R}$ if and only if
    $geod(i,I)\cap \Gamma=\emptyset$.

    {\bf Case 2: } $i\in I^+$. Suppose that $(I,I^+)\in \Gamma$.
    Then there exists $j_0\in I^-$ such that $I\in N_{j_0}^-$
    which implies $ \beta_I=\beta_{j_0}$. Hence
    $\beta_{j_0}=\alpha_I$ follows from $\alpha_I=\beta_I$. So $U$
    is singular, a contradiction.
    Therefore $(I,I^+)\notin \Gamma$. By Theorem
     \ref{schur decomposition} (ii), $n\neq i\in \mathcal{R}$ if and only if
$n\neq i\in {\mathcal{R}}_K$. Because  $U_K$ is $ \mathcal{U}$
matrix, $n\neq i\in \mathcal{R}$ if and only if $geod(i,I^+)\cap
 \Gamma=\emptyset$  by the induction hypothesis, so  if and only
 $geod(i,I)\cap \Gamma=\emptyset$.

(ii) Since $U|_L$ is GU matrix or $ \mathcal{U}$ matrix, the
assertion follows from Theorem 3 in \cite{dell2000} or (i).
\end{proof}

 \begin{theorem}
\label{exit of n}
 Let $U$ be  a $\mathcal{U}$ matrix. Then

{\rm{ (i).}} $U^{-1}$ is row diagonal dominant $M-$ matrix if and
only if
\begin{equation}\label{3.1}
\overline{\mu}_n \ge \sum_{L\in geod(r,n),L\neq
n}\frac{\overline{\mu}_{L^-}(1-\alpha_L\overline{\mu}_{L^+})}{
1-\alpha_L\overline{\mu}_{L^-}}.
\end{equation}

{\rm {(ii).}} $n\in \mathcal{R}$ if and only if (\ref{3.1}) becomes strict
inequality.
 \end{theorem}
\begin{proof}
 From Theorem \ref{schur decomposition} (ii),  the sum of $n-$th
 row of $U^{-1}$ is
 $$ (\mu_K)_n-
 \frac{\overline{\mu}_{I^-}(1-\alpha_I\overline{\mu}_{I^+})}{
1-\alpha_I\overline{\mu}_{I^-}},$$ where $(\mu_K)_n$ is the last
component of $\mu_K$. Hence by the inductive hypothesis, the sum
of $n-$th row of $U^{-1}$ is
$$
\overline{\mu}_n - \sum_{L\in geod(r,n),L\neq
n}\frac{\overline{\mu}_{L^-}(1-\alpha_L\overline{\mu}_{L^+})}{
1-\alpha_L\overline{\mu}_{L^-}}.
$$
Therefore (i)  follows from Theorem \ref{schur decomposition}
 (iii). (ii) is just a consequence of (i) and the definition of
 exiting.
  \end{proof}
\begin{lemma}
\label{gu matrix} Let $U=(U_{ij},i,j\in I)$  a nonsingular GU
matrix. Then $ U_{ii}\overline{\mu}_I\ge 1$  for all $i\in I$.
Moreover, $\displaystyle\max\{ U_{ii}, \ i \in I
\}\overline{\mu}_I=1$ if and only if there exist a column whose
entries are equal to $\displaystyle\max\{ U_{ii}, \ i \in I \}$.
\end{lemma}
\begin{proof}
Since $U$ is a GU matrix and ${\bf 1}=UU^{-1}{\bf 1} =U\mu$,
$1=\sum_{j=1}^{n}U_{ij}(\mu_I)_j\le
\sum_{j=1}^{n}U_{ii}(\mu_I)_j$. Hence  we have $
U_{ii}\overline{\mu}_I\ge 1$ for $i\in I.$ Let
$\displaystyle\max\{ U_{ii}\}=U_{qq} $ and suppose that
$U_{qq}\overline{\mu}_I=1.$  Then
$1=\sum_{j=1}^{n}U_{ij}(\mu_I)_j\le
\sum_{j=1}^{n}U_{ii}(\mu_I)_j\le U_{qq}\overline{\mu_I}=1$. Hence
$\sum_{j=1}^{n}(U_{ij}-U_{qq})\mu_j=0$  for $i\in I$, which yields
the result. Conversely, let $\max\{U_{ii}\}=U_{qq}$ and
$e_q=(0,\cdots, \frac{1}{U_{qq}},\cdots, 0)^t$. Then $Ue_q={\bf
1}$ Hence $\mu=e_q$ and $U_{qq}\overline{\mu_I}=1$.
\end{proof}
\begin{theorem}
\label{is not diagonially dominant }
 Let $U$ be a nonsingular $ \mathcal{U} $ matrix. If there exists
  $i$ with $i\neq n$ such that $ U_{ii}< U_{nn}.$  Then $U$ is not  a row diagonally
  dominant matrix, neither  $n$ is  a root of $P$.
  \end{theorem}
\begin{proof}
We prove the assertion by the induction on the dimension of matrix
$U$ . It is easy to show that the assertion holds for order $n=
1,2.$  Now we assume that
$$U=\left(\begin{array}{cc}
   U_J & \alpha_I\textbf{1}_J\textbf{1}_K^t\\
   b_K\textbf{1}_J^t& U_K
   \end{array}\right).$$
If there exists a  $i\in K$ such that $(U)_{ii}< U_{nn}$, then by
the induction hypothesis,  the last component of $\mu_K$ is less
than 0. So $(\mu_I)_n<0$ by theorem \ref{schur decomposition}(ii).
Hence we assume that there exists  a $i\in J$ such that
$U_{ii}<U_{nn}$. Then by Lemma \ref{gu matrix},
$\overline{\mu}_J\ge \frac{1}{(U_J)_{ii}}=\frac{1}{U_{ii}}.$ Hence
\begin{eqnarray*}
\sum_{i\in K}(\mu_I)_i&=& \overline{\mu}_K- \frac{
\overline{\mu}_J(1-\alpha_I\overline{\mu}_K)}{
1-\alpha_I\overline{\mu}_J} =\frac{
\overline{\mu}_K-\overline{\mu}_J}{ 1-\alpha_I\overline{\mu}_J}\\
&\le & \frac{ 1}{
1-\alpha_I\overline{\mu}_J}(\frac{1}{U_{nn}}-\frac{1}{U_{ii}})<0,
\end{eqnarray*}
since $\overline{\mu}_K=\overline{\mu}_I=\frac{1}{U_{nn}}$ by
Theorem \ref{schur decomposition} (iii) and (v).  On the other
hand, $(\mu_I)_j\ge 0$ for $i\in K\setminus\{n\}$. Therefore,
$(\mu_I)_n<0$. So the assertion holds.
\end{proof}
\begin{corollary}
\label{n is not exiting of P}
  Let $U$ be a nonsingular $
\mathcal{U} $ matrix. If there exists a $i$ with
 $i\neq n$ such that $ U_{ii}\le U_{nn}$  then $n$ is not root of $P$.
 \end{corollary}
\begin{proof}
It follows from Theorem \ref{is not diagonially dominant } and its
proof.
\end{proof}
\begin{lemma}
\label{the smallese row sum}
 Let $U$ be a $ \mathcal{U}$ matrix. If $\sum_{j\in I}U_{nj} \le
 \sum_{j\in I}U_{ij}$ for $i=1,2,\cdots, n-1$. Then $U$ is  a row and
 column diagonally dominant  M-matrix and $n$ is an exiting of $P$.
 \end{lemma}
 \begin{proof}
From $U^{-1}U=I_n,$ $ 1=\sum_{j\in I}\sum_{l\in
I}(U^{-1})_{nl}U_{lj}= \sum_{l\in I,\ l\neq n}\sum_{j\in
I}U_{lj}(U^{-1})_{nl}+ \sum_{j\in I}U_{nj}(U^{-1})_{nn}$. Since
$(U^{-1})_{nl}\le 0$ for $l\neq n$ and $\sum_{j\in I}U_{nj} \le
 \sum_{j\in I}U_{lj}$ for $l\neq n, $   $1\le
\sum_{l\in I,\ l\neq n}\sum_{j\in I}U_{nj}(U^{-1})_{nl}+
\sum_{j\in I}U_{nj}(U^{-1})_{nn}= \sum_{j\in I}U_{nj}\sum_{l\in I
}(U^{-1})_{nl}.$ Hence $\sum_{l\in I }(U^{-1})_{nl}\ge \frac{1}{
\sum_{j\in I}U_{nj}}>0$. Hence the result follows form theorem
\ref{schur decomposition}.
 \end{proof}

\section{Links of P }
\noindent
In this section, we describe completely  the links of  transient
kernel $P$ associated with  a class of  $ \mathcal{U}$ matrices.
Firstly we give some lemmas.

\begin{lemma}
\label{links of I-} Let $U$ be a nonsingular $ \mathcal{U}$
matrix. Then for $ i,j\in I^{-}=J,\ i\neq j$, $ (U^{-1})_{ij}<0$
if and only if $ (U_{J}^{-1})_{ij}<0$ and $ U_{ij}>\alpha_I$.
\end{lemma}
\begin{proof} Necessity.
By Theorem \ref{schur decomposition},  we have
$$(U^{-1})_J=U_J^{-1}+\frac{\alpha_I}{1-\alpha_I\overline{\mu}_J}\mu_J\nu_J^t=(U_J(\alpha_I))^{-1},$$
 where
 $U_J(\alpha_I)=U_J-\alpha_I\textbf{1}_J\textbf{1}_J^t$ is a nonsingular GU matrix.
 Hence for $i,j\in I^-=J,$  $(U^{-1})_{ij}<0$ implies  $ (U_{J}^{-1})_{ij}<0$.
 Further, by Theorem 3.6 in \cite{nabben1995},
 $(U^{-1})_{ij}=(U_J(\alpha_I))^{-1}_{ij}<0 $ implies that $(U_J(\alpha_I))_{ij} >
 0.
 $  So $ U_{ij}>\alpha_I$.

 Sufficiency. Note that $ \alpha_I< \displaystyle \min\{
 (U_J)_{ii}\}$ (otherwise every entries of some row  rather than $n$  is
 $\alpha_I$ which yields $U$ is singular.)  Note that $U_J$ is a $GU$ matrix.
  By Theorem 6
 in \cite{dell2000}, $ (U_J^{-1})_{ij}<0 $ implies that $(U^{-1})_{ij}=(U_J(\alpha_I)^{-1})_{ij}<0
 $.
\end{proof}
\begin{lemma}
\label{link for I-I+} Let $U$ be a nonsingular $ \mathcal{U}$
matrix of order $n$.  Then

{\rm{(i).}} for $i\in I^{-}=J,\ j\in I^+=K $, $(U^{-1})_{ij}<0 $
if and only if $i\in {\mathcal{R}}_J$ and $j=n$.

{\rm{(ii).}}  for $i\in I^{+},\ j\in I^- $, $(U^{-1})_{ij}<0 $ if
and only if  $i=n$  and $j\in {\mathcal{R}}_J$.
\end{lemma}
\begin{proof}
(i).  For $i\in I^{-}=J,\ j\in I^+=K $,  by Theorem \ref{schur
decomposition} (i), $
(U^{-1})_{ij}=(\frac{-\alpha_I}{1-\alpha_I\overline{\mu}_J}\mu_J\nu_K^t)_{ij}<0$
if and only if $(\mu_J)_i>0$ and $(\nu_K)_j>0$ if and only if
$i\in {\mathcal{R}}_J$ and $j=n$.

(ii). For $i\in I^{-}=J,\ j\in I^+=K $,  by Theorem \ref{schur
decomposition},  $(U^{-1})_{ij}<0 $ if and only if $(
\frac{-1}{1-\alpha_I\overline{\mu}_J}e_K\nu_J^t)_{ij}<0$  if and
only if  $i=n$  and $j\in {\mathcal{R}}_J$.
\end{proof}
\begin{lemma}
\label{links for I+I+} Let $U$ be a nonsingular  $ \mathcal{U}$
matrix. Then for $i\in I^{+}=K,\ j\in I^+ $ and $i\neq j$,
$(U^{-1})_{ij}<0 $ if and only if $(U_K^{-1})_{ij}<0$.
\end{lemma}
\begin{proof}
The assertion follows from Theorem \ref{schur decomposition} (V).
\end{proof}

Now we can state  the main result in this section.
\begin{theorem}
\label{link of p} Let $U$ be a nonsingular $ \mathcal{U} $ matrix
associated  supporting tree $(T, \mathcal{T})$ with fixed the leaf
$n$. Suppose that $i\neq j$ and $i\wedge j=L$.

{\rm{(i)}}. If $L\in geod(I,n)$, then $P_{ij}>0$ if and only if
$(P^L)_{ij}>0$; i.e., if and only if

{\rm{(i.a).}} for $i\in L^{-},\ j\in L^+ $,  $i\in
{\mathcal{R}}_{L^-}$ and $j=n$.

{\rm{(i.b).}}  for $i\in I^{+},\ j\in I^- $,  $i=n$ and $j\in
{\mathcal{R}}_{L^-}$.

{\rm{(ii)}} If $L\notin geod(I,n)$ and $L_1= (i \wedge j)\wedge
n$, then

{\rm{(ii.a).}} $(P^L)_{ij}>0 $ if and only if $i\in
{\mathcal{R}}_{L^-}$ and
 $ j\in {\mathcal{R}}_{L^+}$ for $ i\in L^-$, $ j\in L^+$;  and
$i\in {\mathcal{R}}_{L^+}$ and
 $ j\in {\mathcal{R}}_{L^-}$ for $ i\in L^+$, $ j\in L^-$.

{\rm{ (ii.b).}}  $(P^{L_1^-})_{ij}>0 $ if and only if
$(P^L)_{ij}>0 $;  and
 either $U_{ij}>\alpha_L$, or $U_{ij}=\alpha_{ij}$ and for every
 $M\prec L$ such that $ \alpha_M=\alpha_L $ implies that
 $(M,M^-)\notin \Gamma^t$ for $\{i,j \}\subseteq M^-$,
$(M,M^+)\notin \Gamma$ for $\{i,j \}\subseteq M^+$ hold.

{\rm{(ii.c).}} $P_{ij}>0$ if and only if $(P^{L_1^-})_{ij}>0$ and
$U_{ij}>\alpha_{L_1}$.
 \end{theorem}
 \begin{proof}
We prove the assertion by the dimension of the matrix $U$. It is
trivial for $n=1,2$.  Now assume that $i\neq j$ and $i\wedge j=L$.

 {\bf Case 1: } $L\in geod(I,n)$. If $L=I$, then $i\in I^-$ (or $I^+$) and $j\in I^+$
  ( or $I^-$). Hence by Lemma
  \ref{link for I-I+} and (2), the assertion (i) holds. If $L\succ I$, then $L\in
 I^+$ and  $i,j \in I^+$.  So by Lemma \ref{links for I+I+}, we have $P_{ij}>0$ if and
 only if $(U^{-1})_{ij}<0$ if and only if $(U_{I^+}^{-1})_{ij}<0$.
 Since $U_{I^+}$ is  a $ \mathcal{U}$ matrix, by the inductive  hypothesis,
 we have $(U^{-1})_{ij}<0$ if and only if
$(P^L)_{ij}>0$. Moreover, the rest of
    (i) follows from Lemma \ref{link for I-I+}.

{\bf Case 2:      } $L\notin geod(I,n)$. If $L_1=I,$ then $i,j\in
L\subseteq I^-$ and $U_J$ is a GU matrix. Hence  (ii.a) and (ii.b)
follow from Theorem 4 in \cite{dell2000}. At the same time, (ii.c)
follows from  Lemma \ref{links of I-} and (2) that
 $P_{ij}>0 $ if
and only if $ (U_{I^-}^{-1})_{ij}<0$ and $ U_{ij}>\alpha_{L_1}$.
If $L_1\succ I$, then $P_{ij}>0$ if and
 only if $(U^{-1})_{ij}<0$ if and only if $(U_{I^+}^{-1})_{ij}<0$ if and
 only if $(P^{I^+})_{ij}>0$. By the inductive hypothesis and
 Theorem 4 in \cite{dell2000}, the assertions of (i)  and (ii) hold.
 \end{proof}
\begin{corollary}
\label{zero part} Let $U$ be a nonsingular $ \mathcal{U}$ matrix.
If $A, B\in geod(I,n)$ and $A\neq B\neq n$, then $(U^{-1})_{ij}=0$
for $i\in A$ and $j\in B$.
\end{corollary}
\begin{proof}
Since $A\wedge B=A$  or $B$, The result follows from Theorem
\ref{link of p} (i).
\end{proof}
\begin{corollary}
\label{form of inverse} Let $U$ be nonsingular $ \mathcal{U}$
matrix with supporting tree. If  $\beta_i=\beta_t$ for all $t\prec
i$ and all $i\in I$, then the inverse of $U$ has the following
structure:
$$U^{-1}=\left(\begin{array}{lllll}
W_{11}& 0& \cdots & 0& W_{1s}\\
0& W_{22}& \cdots &0& W_{2s}\\
 \cdots &\cdots &\cdots &\cdots &\cdots \\
W_{s1}& W_{s2}& \cdots &W_{s,s-1}& W_{ss}
\end{array}
\right),$$
 where  $W_{ii}$ is a lower triangular matrix for $i=2,3\cdots, s-1$;
  $W_{ss}$ is a $1\times 1$ matrix.
Moreover, if  $\beta_i>\beta_t$ for all $t\prec i$ and all $i\in
I$; and both $\alpha_{I^-}>\alpha_I$ and $\beta_{A^-}>\beta_A$ for
$A\in geod(I,n)$; then each entry of $W_{11}$ is nonzero, each
entry of $W_{is}$ and $W_{si}$ is nonzero for $i=1,\cdots, s$; and
each each entry of the lower part of lower triangular matrix
$W_{ii}$ is zero for $i=2,\cdots, s-1$.

\end{corollary}
\begin{proof}  We partition the blocks of $U^{-1}=(W_{ij})$  corresponding to the leaves
 sets of vertices of $geod(I,n)$. In particular,  $W_{ss}$ is
 corresponding to fixed vertex $n$. By Corollary
 \ref{zero part}, $W_{ij}=0$ for $i\neq j\neq s$.
Further, it follows from Theorem \ref{link of p}(ii.c) that
$W_{ii}$ is a lower triangular matrix for $i=2,\cdots, s-1$, since
for $A\in I^+$, $\alpha_A=\alpha_{A\wedge n}$. Let $n\neq A\in
geod(A,n)$. Since $\beta_i>\beta_t$ for all $t\prec i$ and all
$i\in I$, $U|_A{^-}$ is a strictly generalized ultrametric matrix.
Hence by Theorem \ref{link of p}(ii.a) and (ii.b) or Theorem 3.5
in \cite{nabben1995}, each entry of $W_{11}$ is nonzero, since
$\alpha_{I^-}>\alpha_I$; and each entry of the lower part of the
lower triangular matrix $W_{ii}$ is nonzero for $i=2,\cdots, s-1$,
since $\beta_{A^-}>\beta_A=\alpha_A$. Moreover, since $ {\mathcal
{R}}_{A^-}=A^-$, each entry of $W_{is}$ and $W_{si}$ is nonzero by
Theorem \ref{link of p} (i). The proof is finished.\end{proof}

\begin{remark}
 From Theorem \ref{link of p},  it is easy to see that the links
 of $U\in \mathcal{U} $ are not involved in whether $n\in \mathcal{R}$
 or not. Hence we may directly determine whether each entries of
 $U^{-1}$ is zero or not from the structure of support tree.
   Let us to give an example to illustrate
Theorems \ref{ root of P} \ref{root of transpose of P}, \ref{link
of p}
\end{remark}
\begin{example} Let $U$ be a $ \mathcal{U}$ matrix of order
$7$ with support tree $(T,\mathcal{J})$ as in the Figure 1, where
$I$ is root and $6$ is fixed leaf.
\end{example}

 \setlength{\unitlength}{0.1in}
\begin{picture}(60,25)
\put(20,5){\circle{0.5}}
 \put(20.16,5.13){\line(2,1){9.9}}
\put(19.85,5.12){\line(-2,1){9.9}}
 \put(9.7,10.3){\circle{0.5}}
\put(30.25,10.3){\circle{0.5}}
 \put(30,10.5){\line(-2,1){7}}
 \put(30.6, 10.4){\line(2,1){7}}
 \put(22.75,14){\circle{0.5}}
 \put(37.76,14){\circle{0.5}}
  \put(22.5,14.2){\line(-1,1){4}}
   \put(38,14.2){\line(1,1){4}}
 \put(22.85, 14.2){\line(1,1){4}}
 \put(37.6,14.2){\line(-1,1){4}}
\put(18.5,18.5){\circle{0.5}} \put(27,18.5){\circle{0.5}}
\put(33.5,18.5){\circle{0.5}} \put(42,18.5){\circle{0.5}}
 \put(9.6,10.5){\line(-2,3){5.1}}\put(9.9,10.5){\line(2,3){5.1}}
  \put(4.5,18.4){\circle{0.5}} \put(15.1,18.5){\circle{0.5}}
   \put(18.5,2.5){$I (1,1)$} \put(3,10){$A (2,3)$}
   \put(2,20){$1 (3,3)$}\put(12,20){$2 (3,3)$} \put(31 ,9.25){$B (2,2)$}
   \put(17,13.4){$C(2,4)$}\put(39,13.4){$D(3,3)$}\put(18,20){$3(4,4)$}
   \put(26,20){$4(4,4)$}\put(32,20){$5(4,4)$}\put(40,20){$6(4,4)$}
   \put(13.4,6){$-$}\put(4,14){$-$}\put(13,14){$+$}\put(26,6){$+$}\put(24,11.52){$-$}\put(36,11.52){$+$}
\put(18,16.2){$-$}\put(26.5,16.2){$+$}\put(32.5,16.2){$-$}\put(41.52,16.2){$+$}
 \put(25,0){\bf Fig. 1}
          \end{picture}

Then  the matrix $U$ and inverse of $U$ are

$$U=\left(\begin{array}{cccccc}
3&2&1&1&1&1\\
3&3&1&1&1&1\\
2&2&4&2&2&2\\
2&2&4&4&2&2\\
3&3&3&3&4&3\\
4&4&4&4&4&4
\end{array}
\right)\ \ \
  U^{-1}=\frac{1}{8}\left(\begin{array}{cccccc}
8&-4&0&0&0&-1\\
-8&8&0&0&0&0\\
0&0&4&0&0&-2\\
0&0&-4&4&0&0\\
0&0&0&0&8&-6\\
0&-4&0&-4&-8&11
\end{array}
\right).
$$
It is easy to see that $\Gamma=\{(A,2), (C,4) \}$ and $
\Gamma^t=\{(A,1),(C,3),(I,A),(B,C),(D,5) \}$. By Theorem
  \ref{ root of P} and Corollary \ref{exit of n}, we have
${\mathcal{R}}=\{1, 3, 5\}$. Further, we determine  all links of
$P$ by Theorem \ref{link of p}. For instance, in order to
determine $P_{43},$ we  consider   $3\wedge 4= C\notin geod(I, 6)$
and $(3\wedge 4) \wedge 6=B$. By Theorems \ref{ root of P} and
\ref{root of transpose of P}, $4\in {\mathcal R}_C$, $3\in {\mathcal
R}_{C^+}^t$.  Hence by Theorem \ref{link of p} (ii.a), $
(P^C)_{43}>0$.  Further, by Theorem \ref{link of p}(ii.b) and
$U_{43}=4>\alpha_C=2$, $(P^{B^-})_{43}>0$. Therefore $P_{43}>0$
follows from Theorem \ref{link of p}(ii.c) and
$U_{43}=4>\alpha_B=2$.
\\
\\
{\bf Acknowledgement.} The research is supported by National Natural Science Foundation of China (No.11271256), Innovation Program of Shanghai Municipal Education Commission (No.14ZZ016)and Specialized Research Fund for the Doctoral Program of Higher Education (No.20130073110075). The author would like to thank 
anonymous referees for their  comments and suggestions.

\bibliographystyle{amsplain}

\end{document}